\documentclass[A4,12pt,reqno]{amsart}
\usepackage{booktabs}
\usepackage{blkarray}
\usepackage{caption}
\usepackage{amsthm}
\usepackage{graphicx}

\newtheoremstyle{definition}
  {15pt}  
  {15pt}  
  {\rm}   
  {\parindent}      
  {\bf}   
  {. }    
  { }     
  {}      
\theoremstyle{definition}
\newtheorem{definition}{Definition}[section]

\newtheorem{remark}{Remark}[section]

\newtheoremstyle{theorem}
  {15pt}  
  {15pt}  
  {\sl}   
  {\parindent}      
  {\bf}   
  {. }    
  { }     
  {}      
\theoremstyle{theorem}
\newtheorem{theorem}{Theorem}[section]
\newtheorem{lemma}{Lemma}[section]
\newtheorem{corollary}{Corollary}[section]

\allowdisplaybreaks 

\begin{document}
\title[Moving Least-Squares Approximation]{Matrices Associated with Moving Least-Squares Approximation and Corresponding Inequalities}
\author[S. Nenov]{Svetoslav Nenov}
\email[Corresponding author]{nenov@uctm.edu}
\author[T. Tsvetkov]{Tsvetelin Tsvetkov}
\email{tstsvetkov@uctm.edu}
\address{Department of Mathematics, University of Chemical Technology and Metallurgy, Sofia 1756, Bulgaria}

\begin{abstract}
In this article, some properties of matrices of moving least-squares approximation have been proven.
The used technique is based on singular-value decomposition and inequalities for singular-values.
Some inequalities for the norm of coefficients-vector of the linear approximation have been proven.
\end{abstract}

\maketitle

\section{Statement}

Let us remind the definition of moving least-squares approximation and a basic result.

Let:
\begin{enumerate}
\item $\mathcal D$ be a bounded domain in $\mathbb R^d$.
\item $\boldsymbol x_i\in \mathcal D$, $i=1,\dots,m$; $\boldsymbol x_i\not=\boldsymbol x_j$, if $i\not=j$.
\item $f: \mathcal D\to \mathbb R$ be a continuous function.
\item $p_i: \mathcal D\to \mathbb R$ be continuous functions, $i=1,\dots,l$. The functions $\{p_1,\dots,p_l\}$ are linearly independent in $\mathcal D$ and let $\mathcal P_l$ be their linear span.
\item $W:(0,\infty)\to(0,\infty)$ be a strong positive function.
\end{enumerate}

Usually the basis in $\mathcal P_l$ is constructed by monomials. For example: $p_l(\boldsymbol x)=x_1^{k_1}\dots x_d^{k_d}$, where $\boldsymbol x=(x_1,\dots,x_d)$, $k_1,\dots k_d\in\mathbb N$, $k_1+\dots+k_d\leq l-1$. In the case $d=1$, the standard basis is $\{1,x,\dots,x^{l-1}\}$.

Following \cite{AlexaBehrCohen-Or}, \cite{Levin1}, \cite{Levin2}, \cite{Levin3}, we will use the following definition. The {\it moving least-squares approximation} of order $l$ at a fixed point $\boldsymbol x$ is the value of $p^*(\boldsymbol x)$, where $p^*\in\mathcal P_l$ is minimizing the least-squares error
$$
\sum_{i=1}^{m}W(\|\boldsymbol x-\boldsymbol x_i\|)\left(p(\boldsymbol x)-f(\boldsymbol x_i)\right)^2
$$
among all $p\in\mathcal P_l$.

The approximation is ``local'' if weight function $W$ is fast decreasing as its argument tends to infinity and interpolation is achieved if $W(0)=\infty$. So, we define additional function $w:[0,\infty)\to[0,\infty)$, such taht:
$$w(r)=
\begin{cases}
\dfrac1{W(r)},&\quad \text{if ($r>0$) or ($r=0$ and $W(0)<\infty$)},\\
0,&\quad \text{if ($r=0$ and $W(0)=\infty$)}.
\end{cases}
$$
Some examples of $W(r)$ and $w(r)$, $r\geq 0$:
\begin{alignat*}{2}
&W(r)=e^{-\alpha^2r^2}&&\qquad \text{$\exp$-weight},\\
&W(r)=r^{-\alpha^2}&&\qquad \text{Shepard weights},\\
&w(\boldsymbol x, \boldsymbol x_i)=r^2e^{-\alpha^2r^2}&&\qquad \text{McLain weight},\\
&w(\boldsymbol x,\boldsymbol x_i)=e^{\alpha^2r^2}-1&&\qquad \text{see Levin's works}.
\end{alignat*}

Here and below: $\|\cdot\|=\|\cdot\|_2$ is 2-norm, $\|\cdot\|_1$ is 1-norm in $\mathbb R^d$; the superscript $^t$ denotes transpose of real matrix; $I$ is the identity matrix.

We introduce the  notations:
\begin{align*}
E=& \begin{pmatrix}
p_1(\boldsymbol x_1) & p_2(\boldsymbol x_1) & \cdots & p_l(\boldsymbol x_1)\\
p_1(\boldsymbol x_2) & p_2(\boldsymbol x_2) & \cdots & p_l(\boldsymbol x_2)\\
\vdots   & \vdots   &        & \vdots\\
p_1(\boldsymbol x_m) & p_2(\boldsymbol x_m) & \cdots & p_l(\boldsymbol x_m)\\
\end{pmatrix},\ \boldsymbol a = \begin{pmatrix}
a_1\\
a_2\\
\vdots\\
a_m
\end{pmatrix},\\
D= & 2 \begin{pmatrix}
w(\boldsymbol x,\boldsymbol x_1) & 0 & \cdots & 0\\
0 & w(\boldsymbol x,\boldsymbol x_2) & \cdots & 0\\
\vdots   & \vdots   &        & \vdots\\
0 & 0 & \cdots & w(\boldsymbol x,\boldsymbol x_m)\\
\end{pmatrix},\  \boldsymbol c = \begin{pmatrix}
p_1(\boldsymbol x)\\
p_2(\boldsymbol x)\\
\vdots\\
p_l(\boldsymbol x)
\end{pmatrix}.
\end{align*}

Through the article, we assume the following conditions (H1):
\begin{enumerate}
\item[(H1.1)] $1\in \mathcal P_l$.
\item[(H1.2)] $1\leq l \leq m$.
\item[(H1.3)] rank$(E^t)=l$.
\item[(H1.4)] $w$ is smooth function.
\end{enumerate}


\begin{theorem}[see \cite{Levin1}] Let the conditions (H1) hold true.

Then:
\begin{enumerate}
\item
The matrix $E^tD^{-1}E$
is non-singular.
\item The approximation defined by the moving least-squares method is
\begin{gather}
\hat L (f) = \sum_{i=1}^m a_i f(\boldsymbol x_i),
\end{gather}
where
\begin{gather}
\boldsymbol a = A_0 \boldsymbol c\quad\text{and}\quad   A_0= D^{-1}E\left(E^tD^{-1}E\right)^{-1}.
\end{gather}
\item If $w(\|\boldsymbol x_i-\boldsymbol x_i\|)=0$ for all $i=1,\dots,m$, then the approximation is interpolatory.
\end{enumerate}
\end{theorem}

For the approximation order of moving least-squares approximation (see \cite{Levin1} and \cite{Fasshauer}) it is not difficult to receive (for convenience we suppose $d=1$ and standard polynomial basis, see \cite{Fasshauer}):
\begin{align}
\left|f(x)-\hat L (f)(x)\right|
\leq \|f(x)-p^*(x)\|_{\infty}\left[1+\sum\limits_{i=1}^{m} |a_i|\right],\label{eq:intr:11}
\end{align}
and moreover ($C$=const.)
\begin{gather}\label{eq:intr:12}
\|f(x)-p^*(x)\|_{\infty} \leq C h^{l+1} \max\left\{\left|f^{(l+1)}(x)\right|:x\in\overline{\mathcal D}\right\}.
\end{gather}
It follows from \eqref{eq:intr:11} and \eqref{eq:intr:12} that the error of moving least-squares approximation is upper-bounded from the 2-norm of coefficients of approximation ($\|\boldsymbol a\|_1\leq \sqrt{m}\|\boldsymbol a\|_2$). That is why, the goal in this short note, is to discuss a method for majorization in the form
$$
\|\boldsymbol a\|_2 \leq M \exp\left(N\|\boldsymbol x - \boldsymbol x_i\|\right),
$$
Here the constants $M$ and $N$ depends on singular values of matrix $E^t$, and numbers $m$ and $l$ (see Section 3). In Section 2 some properties of matrices associated with approximation (symmetry, positive semi-definiteness, and norm majorization by $\sigma_{min}(E^t)$ and $\sigma_{max}(E^t)$) are proven. 

The main result in Section 3 is formulated in the case of $\exp$-moving least-squares approximation, but it is not hard to receive analogous results in the different cases: Backus-Gilbert wight functions, McLain wight functions, etc.

\section{Some Auxiliary Lemmas}\label{sec:second}

\begin{definition}
We will call the matrices
$$
A_1 = A_0E^t=D^{-1}E\left(E^tD^{-1}E\right)^{-1}E^t\quad \text{and}\quad A_2=A_1-I
$$
{\it $A_1$-matrix} and {\it $A_2$-matrix} of the approximation $\hat L$, respectively.
\end{definition}

\begin{lemma}\label{lem:2.1} Let the conditions (H1) hold true.

Then, the matrices $A_1D^{-1}$ and $A_2D^{-1}$ are symmetric.
\end{lemma}

\begin{proof}
Direct calculation of the corresponding transpose matrices.
\end{proof}
\begin{lemma}\label{lem:2.2} Let the conditions (H1) hold true.

Then:
\begin{enumerate}
\item All eigenvalues of $A_1$ are 1 and 0 with geometric multiplicity $l$ and $m-l$, respectively.
\item All eigenvalues of $A_2$ are 0 and -1 with geometric multiplicity $l$ and $m-l$, respectively.
\end{enumerate}
\end{lemma}
\begin{proof} Part 1. We will prove that the dimension of the null-space\linebreak $\dim\left(\text{null}\,(A_2)\right)$ is at least $l$.

Using the definition of $A_2=D^{-1}E\left(E^tD^{-1}E\right)^{-1}E^t-I$, we receive
$$
E^tA_2=\left(E^tD^{-1}E\right)\left(E^tD^{-1}E\right)^{-1}E^t - E^t=0.
$$
Hence
$$
\text{im}\,(A_2)\subseteq\text{null}\,(E^t).
$$
Using (H1.3), $E^t$ is $(l\times m)$-matrix with maximal rank $l$ ($l<m$). Therefore $\dim(\text{null}\,(E^t))=m-l$. Moreover $\dim\left(\text{im}\,(A_2)\right)=m-\dim\left(\text{null}\,(A_2)\right)$. That is why $m-\dim\left(\text{null}\,(A_2)\right) \leq m-l$ or $l\leq \dim\left(\text{null}\,(A_2)\right)$.
\bigskip

Part 2. We will prove that $-1$ is eigenvalue of $A_2$ with geometric multiplicity  $m-l$, or the system
$$
A_2\boldsymbol\eta=-\boldsymbol\eta \ \Longleftrightarrow A_1\ \boldsymbol\eta = 0
$$
has $m-l$ linearly independent solutions. 

Obviously the systems
\begin{gather}\label{eq:help:2.1}
A_1\boldsymbol\eta=D^{-1}E\left(E^tD^{-1}E\right)^{-1}E^t \boldsymbol \eta = 0
\end{gather}
and
\begin{gather}\label{eq:help:2.2}
E^t \boldsymbol \eta = 0
\end{gather}
are equivalent. Indeed, if $\boldsymbol \eta_0$ is a solution of \eqref{eq:help:2.1}, then
\begin{align*}
D^{-1}E\left(E^tD^{-1}E\right)^{-1}E^t \boldsymbol \eta_0 = 0\ \Longrightarrow&\ E^tD^{-1}E\left(E^tD^{-1}E\right)^{-1}E^t \boldsymbol \eta_0 = 0\\
\Longrightarrow&\ E^t \boldsymbol \eta_0 = 0,
\end{align*}
i.e. $\boldsymbol \eta_0$ is solution of \eqref{eq:help:2.2}.

On the other hand, if $\boldsymbol \eta_0$ is a solution of \eqref{eq:help:2.2}, then
\begin{align*}
\left(D^{-1}E\left(E^tD^{-1}E\right)^{-1}E^t\right) \boldsymbol \eta_0=
\left(D^{-1}E\left(E^tD^{-1}E\right)^{-1}\right)\left(E^t \boldsymbol \eta_0\right)=0,
\end{align*}
i.e. $\boldsymbol \eta_0$ is solution of \eqref{eq:help:2.1}. Therefore
$$
\dim\left(\text{im}\,\left(A_1\right)\right)=\dim\left(\text{im}\,\left(E^t\right)\right)=m-l.
$$

Part 3. It follows from parts 1 and 2 of the proof that $0$ is an eigenvalue of $A_2$ with multiplicity exactly $l$ and $-1$ is an eigenvalue of $A_2$ with multiplicity exactly $m-l$.

It remains to prove that 1 is eigenvalue of $A_1$ with multiplicity at least $l$, but this is analogous to the proven part 1 or it follows dirctly from the definition of $A_1=A_2+I$.
\end{proof}

The following two results are proven in \cite{LuPearce}.

\begin{theorem}[see \cite{LuPearce}, Theorem 2.2]\label{LuPearce} Suppose $U$, $V$ are $(m\times m)$ Hermitian matrices and either $U$ or $V$ is positive semi-definite. Let
$$
\lambda_1(U)\geq\cdots\geq\lambda_m(U),\quad
\lambda_1(V)\geq\cdots\geq\lambda_m(V)
$$
denote the eigenvalues of $U$ and $V$, respectively.

Let:
\begin{enumerate}
\item $\pi(U)$ is the number of positive eigenvalues of $U$; 
\item $\nu(U)$ is the nubver of negative eigenvalues of $U$;
\item $\xi(U)$ is the number of zero eigenvalues of $U$.
\end{enumerate}

Then:
\begin{enumerate}
\item If $1\leq k \leq \pi(U)$, then
\begin{gather*}
\min\limits_{1\leq i\leq k}\left\{ \lambda_i(U)\lambda_{k+1-i}(V) \right\}\geq \lambda_k(VU)\geq
\max\limits_{k\leq i\leq m}\left\{ \lambda_i(U)\lambda_{m+k-i}(V) \right\}.
\end{gather*}
\item If $\pi(U)< k\leq m-\nu(U)$, then
$$
\lambda_k(V U)=0.
$$
\item If $m-\nu(U)< k\leq m$, then
\begin{gather*}
\min\limits_{1\leq i\leq k}\left\{ \lambda_i(U)\lambda_{m+i-k}(V) \right\}\geq \lambda_k(V U)\geq
\max\limits_{k\leq i\leq m}\left\{ \lambda_i(U)\lambda_{i+1-k}(V) \right\}.
\end{gather*}
\end{enumerate}
\end{theorem}

\begin{corollary}[see \cite{LuPearce}, Corollary 2.4]\label{LuPearce:cor}
Suppose $U$, $V$ are $(m\times m)$ Hermitian positive definite matrices.

Then for any $1\leq k \leq m$
$$
\lambda_1(U)\lambda_1(V) \geq \lambda_k(VU)\geq
\lambda_m(U)\lambda_{m}(V).
$$
\end{corollary}

As a result of Lemma \ref{lem:2.1}, Lemma \ref{lem:2.2} and Theorem \ref{LuPearce}, we may prove the following lemma.
\begin{lemma}\label{cor:2.1} Let the conditions (H1) hold true.
\begin{enumerate}
\item Then $A_1D^{-1}$ and $-A_2D^{-1}$ are symmetric positive semi-definite matrices.
\item The following inequality hods true
$$
\lambda_{\max}(A_1D^{-1})\leq \dfrac{1}{\lambda_{\min}(D)}.
$$
\end{enumerate}
\end{lemma}
\begin{proof} (1) We  apply Theorem \ref{LuPearce}, where
$$
U = D,\quad V = A_1D^{-1}.
$$

Obviously, $U$ is a symmetric positive definite matrix (in fact it is a diagonal matrix). Moreover $\pi(U)=m$, $\mu(U)=\xi(U)=0$, if $\boldsymbol x \not= \boldsymbol x_i$, $i=1,\dots,m$.

The matrix $V$ is symmetric, see Lemma \ref{lem:2.1}.

From the cited theorem, for any index $k$ ($k=1,\dots,m=\pi(U)$) we have
\begin{gather*}
\lambda_k(A_{1})=\lambda_k(A_{1}D^{-1}D)=\lambda_k(V U)
\leq\min\limits_{1\leq i\leq k}\left\{ \lambda_i(U)\lambda_{m+i-k}(V) \right\}.
\end{gather*}
In particular, if $k=m$:
\begin{gather}\label{eq:help:2.5}
\lambda_m(A_{1})
\leq\min\limits_{1\leq i\leq m}\left\{ \lambda_i(U)\lambda_i(V) \right\}.
\end{gather}

Let us suppose that there exists index $i_0$ ($i_0=1,\dots,m-1$) such that
\begin{gather}\label{eq:help:2.6}
\lambda_1(V)\geq\cdots\geq \lambda_{i_o}(V)\geq 0> \lambda_{i_o+1}(V)\geq\cdots\geq\lambda_m(V).
\end{gather}

It fowollws from \eqref{eq:help:2.6} and positive definiteness of $U$, that 
$$
\min\limits_{1\leq i\leq m}\left\{ \lambda_i(U)\lambda_i(V) \right\}\leq \lambda_{i_0+1}(U)\lambda_{i_0+1}(V)<0.
$$
Therefore (see \eqref{eq:help:2.5}) $\lambda_m(A_{1})<0$. This contradiction (see Lemma \ref{lem:2.2}) proves that the matrix $A_{1}D^{-1}$ is positive semi-definite.

If we set $U = D$, $V = -A_2D^{-1}$ then by analogical arguments, we see that the matrix $-A_2D^{-1}$ is positive semi-definite.
\bigskip

(2) From the first statement of Lemma \ref{cor:2.1}, $V=A_1D^{-1}$ is positive semi-definite. Therefore (see Corollary \ref{LuPearce:cor} and Lemma \ref{lem:2.2}):
$$
1\geq\lambda_k(A_{1})=\lambda_k(VU) \geq \max\left\{\lambda_m(U)\lambda_k(V),\lambda_m(V)\lambda_k(U)\right\}
$$
for all $k=1,\dots,m$. Moreover all numbers $\lambda_k(U)$, $\lambda_k(V)$ are non-negative and
$$
\lambda_{\max}(D)=\lambda_1(U)\geq\cdots\geq\lambda_m(U)=\lambda_{\min}(D),\quad
\lambda_1(V)\geq\cdots\geq\lambda_m(V).
$$
Therefore
$$
1\geq \max\left\{\lambda_m(U)\lambda_1(V),\lambda_m(V)\lambda_1(U)\right\},
$$
or
\begin{gather*}
\lambda_{\max}(A_1D^{-1})=\lambda_1(V)\leq \dfrac{1}{\lambda_m(U)}=\dfrac{1}{\lambda_{\min}(D)}.\qedhere
\end{gather*}
\end{proof}

In the following, we will need some results related to inequalities for singular values. So, we will list some necessary inequalities in the next lemma.

\begin{lemma}[see \cite{MerikoskiKumar}, \cite{Jabbari}]\label{lem:help:Jabbari}
Let $U$ be an $(d_1\times d_2)$-matrix, $V$ be an $(d_3\times d_4)$-matrix.

Then:
\begin{alignat}{2}
&\sigma_{\max}(UV)\leq\sigma_{\max}(U)\sigma_{\max}(V),&&\label{lem:help:Jabbari:eq1}\\
&\sigma_{\max}(U^{-1})=\dfrac{1}{\sigma_{\min}(U)},&&\quad \text{if $d_1=d_2$, $\det{U}\not=0$},\label{lem:help:Jabbari:eq2}\\
&\sigma_{\max}(V)\sigma_{\min}(U)\leq\sigma_{\max}(UV),&&\quad \text{if $d_1\geq d_2=d_3$},\label{lem:help:Jabbari:eq3}\\
&\sigma_{\max}(U)\sigma_{\min}(V)\leq\sigma_{\max}(UV),&&\quad \text{if $d_4\geq d_3=d_2$},\label{lem:help:Jabbari:eq4}
\end{alignat}

If $d_1=d_2$ and $U$ is Hermitian matrix, then $\|U\|=\sigma_{\max}(U)$, $\sigma_i(U)=|\lambda_i(U)|$, $i=1,\dots,d_1$.
\end{lemma}

\begin{lemma}\label{lem:2.4} Let the conditions (H1) hold true and let $\boldsymbol x\not=\boldsymbol x_i$, $i=1,\dots,m$.

Then:
\begin{align}
&\|A_1D^{-1}\| \leq\dfrac{1}{\lambda_{\min}(D)},\label{lem:2.4:eq1}\\
&\sigma_{\max}(A_1)\sigma_{\min}(D^{-1})\leq\sigma_{\max}(A_1D^{-1}),\label{lem:2.4:eq2}\\
&1\leq\|A_1\|\leq \sqrt{\dfrac{\sigma_{\max}(D)}{\sigma_{\min}(D)}}.\label{lem:2.4:eq3}
\end{align}
\end{lemma}
\begin{proof} The matrix $A_1D^{-1}$ is simmetric and positive semi-definite (see Lemma \ref{cor:2.1}(1)). Using the second statement of Lemma \ref{cor:2.1} and Lemma \ref{lem:help:Jabbari}, we receive
$$
\|A_1D^{-1}\|=\sigma_{\max}(A_1D^{-1})=\lambda_{\max}\left(A_1D^{-1}\right)\leq\dfrac{1}{\lambda_{\min}(D)}.
$$

The inequality \eqref{lem:2.4:eq2} follows from \eqref{lem:help:Jabbari:eq4} ($d_4=d_3=m$).

From \eqref{lem:2.4:eq2} and \eqref{lem:help:Jabbari:eq2}, we receive
$$
\sigma_{\max}(A_1)\leq\dfrac{\sigma_{\max}(A_1D^{-1})}{\sigma_{\min}(D^{-1})}
=\dfrac{\sigma_{\max}(D)}{\sigma_{\min}(D)}.
$$
Therefore the equality $\|A_1\|=\sqrt{\sigma_{\max}(A_1)}$ implies the right inequality in \eqref{lem:2.4:eq3}.

Using $E^t=E^tA_1$ and inequality \eqref{lem:help:Jabbari:eq1}, we receive
$$
\sigma_{\max}(E^t) \leq \sigma_{\max}(E^t)\sigma_{\max}(A_1),
$$
or $1\leq \sigma_{\max}(A_1)=\|A_1\|^2$, i.e. the left inequality in \eqref{lem:2.4:eq3}.

The lemma has been proved.
\end{proof}

\section{An Inequality for the Norm of Approximation Coefficients}

We will use the following hypotheses:
\begin{enumerate}
\item[H2.1.] The hypotheses (H1) hold true.
\item[H2.2.] $d=1$, $x_1<\cdots <x_m$.
\item[H2.3.] The map $\boldsymbol c$ is $C^1$-smooth in $[x_1,x_m]$.
\item[H2.4.] $w(|x-x_i|)=\exp(\alpha (x-x_i)^2)$, $i=1,\dots,m$.
\end{enumerate}

\begin{theorem}\label{thm:3.3.5a} Let the following conditions hold true:
\begin{enumerate}
\item Hypotheses (H2).
\item Let $x\in[x_1,x_m]$ be a fixed point.
\item The index $k_0\in\{1,\dots,m\}$ is choosen such taht $$|x-x_{k_0}|=\min\{|x-x_i|:i=1,\dots,m\}.$$
\end{enumerate}

Then, there exist constants $M_1,M_2>0$ such that
$$
\left\|\boldsymbol a(x)\right\| \leq \Big(\left\|\boldsymbol a(x_{k_0})\right\| + M_1 |x-x_{k_0}|\Big)\exp\left(M_2|x-x_{k_0}|\right).
$$
\end{theorem}
\begin{proof} Part 1. Let
$$
H = \begin{pmatrix}
2\alpha (x-x_1) & 0 & \cdots & 0\\
0 & 2\alpha (x-x_2) & \cdots & 0\\
\vdots   & \vdots   &        & \vdots\\
0 & 0 & \cdots & 2\alpha (x-x_m)\\
\end{pmatrix},
$$
then
$$
\frac{d D}{dx}  =  HD,\quad
\frac{d D^{-1}}{dx} = -HD^{-1}.
$$

We have (obviously $D=D(x)$, $H=H(x)$, and $\boldsymbol c=\boldsymbol c(x)$)
\begin{align*}
\dfrac{d \boldsymbol a(x)}{dx}  = & \dfrac{d}{dx} \left(D^{-1}E\left(E^tD^{-1}E\right)^{-1} \boldsymbol c\right)\\
= & \left(\dfrac{d}{dx} D^{-1}\right)E\left(E^tD^{-1}E\right)^{-1}\boldsymbol c + D^{-1}E \left(\dfrac{d}{dx} \left(E^tD^{-1}E\right)^{-1}\right)\boldsymbol c\\
& + D^{-1}E\left(E^tD^{-1}E\right)^{-1} \dfrac{d}{dx}\boldsymbol c\\
= & - HD^{-1}E \left(E^tD^{-1}E\right)^{-1}\boldsymbol c\\
  & + D^{-1}E \left(-\left(E^tD^{-1}E\right)^{-1} \left(\dfrac{d}{d\alpha} E^tD^{-1}E\right)\left(E^tD^{-1}E\right)^{-1} \right)\boldsymbol c\\
  & + D^{-1}E\left(E^tD^{-1}E\right)^{-1} \dfrac{d}{dx}\boldsymbol c\\
= & - H \boldsymbol a\\
  & + D^{-1}E \left(E^tD^{-1}E\right)^{-1} \left(E^t H D^{-1}E\right)\left(E^tD^{-1}E\right)^{-1} \boldsymbol c\\
  & + D^{-1}E\left(E^tD^{-1}E\right)^{-1} \dfrac{d}{dx}\boldsymbol c\\
= & \left(D^{-1}E \left(E^tD^{-1}E\right)^{-1} E^t-I\right) H\boldsymbol a\\
& + D^{-1}E\left(E^tD^{-1}E\right)^{-1} \dfrac{d}{dx}\boldsymbol c\\
= & A_2 H \boldsymbol a + A_0 \dfrac{d}{dx}\boldsymbol c.
\end{align*}

Therefore, the function $\boldsymbol a(x)$ satisfies the differential equation
\begin{gather}\label{thm:3.3.5a:eq10}
\dfrac{d \boldsymbol a(x)}{dx}  = A_2 H \boldsymbol a + A_0 \dfrac{d}{dx}\boldsymbol c.
\end{gather}

Part 2. Obviously
$$
\|A_2H\|=\left\|(A_1-I)H\right\|\leq(\|A_1\|+1)\|H\|.
$$
It follows from \eqref{lem:2.4:eq3} that
$$
\|A_1\|\leq \sqrt{\dfrac{\sigma_{\max}(D)}{\sigma_{\min}(D)}}.
$$
Here $\sigma_{\max}(D)\leq 2\exp(\alpha r^2)$, $r=x_m-x_1$, and $\sigma_{\min}(D)\geq 2$. Hence
$$
\|A_1\|\leq \sqrt{\exp(\alpha r^2)}.
$$

For the norm of diagonal matrix $H$, we receive
$$
\|H\|\leq 2 \alpha r.
$$

Therefore $\|A_2H\|\leq M_2$, where
$$
M_2= 2 \alpha r\left(1+\sqrt{\exp(\alpha r^2)}\right).
$$

We will use Lemma \ref{lem:help:Jabbari} to obtain the norm of $A_0$.

Obviously $A_0E^t=A_1$. Therefore by \eqref{lem:help:Jabbari:eq4} ($m=d_4\geq d_3=l$) we have
$$
\sigma_{\max}(A_0)\sigma_{\min}(E^t)\leq\sigma_{\max}(A_1),
$$
i.e.
$$
\left\|A_0\right\|\leq\dfrac{1}{\sigma_{\min}(E^t)}\sqrt{\dfrac{\sigma_{\max}(D)}{\sigma_{\min}(D)}}.
$$

Therefore, if we set $M_{11}=\dfrac{M_2}{\sigma_{\min}(E^t)}$, then $\|A_0\|\leq M_1$.

Let the constant $M_{12}$ is choosen such that
$$
\left\|\dfrac{d}{dx}\boldsymbol c(x)\right\|\leq M_{12},\quad x\in[x_1,x_m]
$$
and let $M_1=M_{11}M_{12}$.

\bigskip 

Part 3. On the end, we have only to apply Lemma 4.1 form \cite{Hartman} to the equation \eqref{thm:3.3.5a:eq10}:
\begin{align*}
\|\boldsymbol a(x)\| \leq& \left(\|\boldsymbol a(x_{k_0})\|+\left|\ \int\limits_{x_{k_0}}^{x}\left\|A_0 \dfrac{d}{dx}\boldsymbol c\right\|dx\right|\right)\exp\left|\ \int\limits_{x_{k_0}}^{x}\|A_2H\|dx\right|\\
\leq & \left(\|\boldsymbol a(x_{k_0})\|+M_1|x-x_{k_0}|\right)\exp\left(M_2|x-x_{k_0}|\right).\qedhere
\end{align*}
\end{proof}

\begin{remark}
Let the hypotheses (H2) hold true and let moreover
$$
p_1(x)=1,\ p_2(x)=x,\ \dots,\ p_l(x)=x^{l-1},\quad l\geq 1.
$$

In such a case, we may replace the differentiation of vector-fuction
$$
\boldsymbol c(x)= \begin{pmatrix}
p_1(x)\\
p_2(x)\\
\vdots\\
p_l(x)
\end{pmatrix}=
\begin{pmatrix}
1\\
x\\
\vdots\\
x^{l-1}
\end{pmatrix}
$$
by left-multiplication:
\begin{align*}
\dfrac{d\boldsymbol c(x)}{dx}&=
\begin{pmatrix}
0\\
1\\
2x\\
3x^2\\
\vdots\\
(l-2)x^{l-3}\\
(l-1) x^{l-2}
\end{pmatrix}=
\begin{pmatrix}
0 & 0 & 0 & \dots & 0&0&0\\
1 & 0 & 0 & \dots & 0&0&0\\
0 & 2 & 0 & \dots & 0&0&0\\
0 & 0 & 3 & \dots & 0&0&0\\
\vdots & \vdots & \vdots & & \vdots&\vdots&\vdots\\
0 & 0 & 0 & \dots &l-2& 0&0\\
0 & 0 & 0 & \dots &0& l-1&0\\
\end{pmatrix}
\begin{pmatrix}
1\\
x\\
x^2\\
\vdots\\
x^{l-2}\\
x^{l-1}
\end{pmatrix}\\&=\bar\partial \boldsymbol c(x).
\end{align*}

The singular values of the matrix $\bar\partial$ are: $0,1,\dots,l-1$. Therefore $\|\bar\partial\|=\sqrt{l-1}$.

That is why, we may chose
$$
M_{22}=\sqrt{(l-1)}\max\limits_{1\leq i \leq l}\left\{\max\limits_{x_1<x<x_m}\left|p_i(x)\right|\right\}.
$$

Additionally, if we supose $|x_1|\leq |x_m|$, then
$$
\max\limits_{x_1<x<x_m}\left|p_i(x)\right|=|p_i(x_m)|,\quad i=1,\dots,l.
$$

Therefore, in such a case:
$$
M_{22}=\sqrt{(l-1)}\max\limits_{1\leq i \leq l}\left\{\left|p_i(x_m)\right|\right\}.
$$

If we suppose $-1 \leq x_1 \leq x\leq x_m \leq 1$, then obviously, we may set
$$
M_{22} = \sqrt{l-1}.
$$
\end{remark}

\end{document}